\documentclass[a4paper,11pt]{amsart}

\usepackage{amsfonts,amssymb}
\usepackage{graphicx,tikz}
\usepackage{float}

\theoremstyle{plain}

\newtheorem*{thmA}{Theorem A}
\newtheorem*{thmB}{Theorem B}
\newtheorem{thm}{Theorem}[section]
\newtheorem{lem}[thm]{Lemma}
\newtheorem{pro}[thm]{Proposition}

\theoremstyle{definition}

\newtheorem{exa}[thm]{Example}

\newtheorem{rmk}[thm]{Remark}

\newcommand{\Z}{\mathbb{Z}}

\newcommand{\N}{\mathbb{N}}

\newcommand{\UU}{\mathcal{U}}

\newcommand{\FF}{\mathcal{F}}

\DeclareMathOperator{\Aut}{Aut}
\DeclareMathOperator{\imm}{im}
\DeclareMathOperator{\mni}{mni}
\DeclareMathOperator{\mci}{mci}
\DeclareMathOperator{\Cl}{Cl}

\begin{document}

\title[Normalizer or centralizer restrictions]
{Some restrictions on normalizers or centralizers in finite $p$-groups}
\author[G.A. Fern\'andez-Alcober]{Gustavo A. Fern\'andez-Alcober}
\address{Matematika Saila\\ Euskal Herriko Unibertsitatea UPV/EHU\\
48080 Bilbao\\ Spain}
\email{gustavo.fernandez@ehu.es}
\author[L. Legarreta]{Leire Legarreta}
\address{Matematika Saila\\ Euskal Herriko Unibertsitatea UPV/EHU\\
48080 Bilbao\\ Spain}
\email{leire.legarreta@ehu.es}
\author[A. Tortora]{Antonio Tortora}
\address{Dipartimento di Matematica\\ Universit\`a di Salerno\\
Via Giovanni Paolo II, 132\\ 84084 Fisciano (SA)\\ Italy}
\email{antortora@unisa.it}
\author[M. Tota]{Maria Tota}
\address{Dipartimento di Matematica\\ Universit\`a di Salerno\\
Via Giovanni Paolo II, 132\\ 84084 Fisciano (SA)\\ Italy}
\email{mtota@unisa.it}

\thanks{The first two authors are supported by the Spanish Government, grant
MTM2011-28229-C02-02, and by the Basque Government, grants IT460-10 and IT753-13.
The last two authors would like to thank the Department of Mathematics at the University of the Basque Country for its excellent hospitality while part of this paper was being written. They also wish to thank G.N.S.A.G.A. (INdAM) for financial support.}

\begin{abstract}
We study three restrictions on normalizers or centralizers in finite $p$-groups, namely:
(i) $|N_G(H):H|\le p^k$ for every $H\not\trianglelefteq G$,
(ii) $|N_G(\langle g \rangle):\langle g \rangle|\le p^k$ for every
$\langle g \rangle \not\trianglelefteq G$, and
(iii) $|C_G(g):\langle g \rangle|\le p^k$ for every
$\langle g \rangle \not\trianglelefteq G$.
We prove that (i) and (ii) are equivalent, and that the order of a non-Dedekind finite
$p$-group satisfying any of these three conditions is bounded for $p>2$.
More precisely, we get the best possible bound for the order of $G$ in all three cases, which is $|G|\le p^{2k+2}$.
The order of the group cannot be bounded for $p=2$, but we are able to identify two infinite families of $2$-groups out of which $|G|\le 2^{f(k)}$ for some function $f(k)$ depending only on $k$.
\end{abstract}

\maketitle

\section{Introduction}

The analysis of groups which satisfy some restriction related to normality
is a common topic in group theory.
Classical examples are the determination by Dedekind \cite{ded} and Baer \cite{bae}
of the groups with all subgroups normal (now known as \emph{Dedekind groups\/}), the characterisation by Neumann \cite{neu} of the groups $G$ with $|G:N_G(H)|<\infty$ for every subgroup $H$ as the central-by-finite groups, or the characterisation in the same paper of the groups with $|H^G:H|<\infty$ for every subgroup $H$ as the groups with
finite derived subgroup.
Numerous papers have been devoted to other types of normality conditions, and this is an active area of research nowadays.

In the recent papers \cite{zha-gao} and \cite{zha-guo}, a new condition has been considered in connection to normality, in the realm of nilpotent groups.
If $G$ is nilpotent and $H$ is a proper subgroup of $G$, we know that $|N_G(H):H|>1$.
The normalizer $N_G(H)$ will be as large as the whole of $G$ if $H$ is normal in $G$, but
what happens if we impose a bound to the index $|N_G(H):H|$ for
\emph{every non-normal subgroup\/} $H$?
Dedekind groups satisfy this type of condition vacuously; what can be said about non-Dedekind groups?
For \emph{finite\/} nilpotent groups, this problem reduces to finite $p$-groups, for $p$ a prime.
Thus the following question arises: if $k$ is a fixed positive integer, what can be said about the finite $p$-groups $G$ which satisfy that
\begin{equation}
\label{condition MNI}
|N_G(H):H|\le p^k
\quad
\text{for every $H\not\trianglelefteq G$,}
\end{equation}
and which are not Dedekind groups?

In \cite{zha-gao}, Q.\ Zhang and Gao have classified all such groups for $k=1$, thus
answering a question of Berkovich \cite[Problem 116 (i)]{ber}.
Apart from the non-abelian groups of order $p^3$, we have the group given by the presentation
\[
\langle a,b \mid a^{p^2}=b^{p^2}=1,\ a^b=a^{1+p} \rangle,
\]
the three infinite families of $2$-groups of maximal class (dihedral,
semidihedral, and generalised quaternion), and these two other
families of $2$-groups:
\begin{equation}
\label{k=1 1st infinite family}
\langle a,b \mid a^{2^{n-2}}=b^4=1,\ a^b=a^{-1} \rangle,
\quad
\text{for $n\ge 4$,}
\end{equation}
and
\begin{equation}
\label{k=1 2nd infinite family}
\langle a,b \mid a^{2^{n-2}}=b^4=1,\ a^b=a^{-1+2^{n-3}} \rangle,
\quad
\text{for $n\ge 5$.}
\end{equation}
Observe that the order of these groups is at most $p^4$ for odd $p$, but can be arbitrarily large for $p=2$.
This different behaviour of the odd primes and the even prime is not particular to the case $k=1$.
As X.\ Zhang and Guo have shown in \cite{zha-guo}, for $p>2$ and arbitrary $k$, the order of the non-Dedekind (i.e. non-abelian) groups satisfying (\ref{condition MNI}) is bounded.
More precisely, they get the bound $|G|\le p^{(2k+1)(k+1)}$, which is valid under the seemingly weaker assumption that
\begin{equation}
\label{condition MNI star}
|N_G(\langle g \rangle):\langle g \rangle|\le p^k
\quad
\text{for every $\langle g \rangle\not\trianglelefteq G$.}
\end{equation}

A related problem can be raised about centralizers of elements.
If $G$ is a group and $g\in G$, we have the inclusion
$\langle g \rangle \le C_G(g)$.
If $g\in Z(G)$ then $C_G(g)$ is as large as the whole group $G$, but otherwise we can require that the index $|C_G(g):\langle g \rangle|$ should be small.
Thus we may ask what can be said about $G$ if $|C_G(g):\langle g \rangle|$
is bounded as $g$ runs over $G\smallsetminus Z(G)$, a question that we have addressed
in \cite{FLTT} in the case of finite groups.
In particular, if $G$ is a non-abelian finite $p$-group such that
\[
|C_G(g):\langle g \rangle|\le p^k
\quad
\text{for every $g\in G\smallsetminus Z(G)$,}
\]
we have proved that $|G|\le p^{2k+2}$ with the only exception of $Q_8$, and that this bound is sharp.
By similarity with condition (\ref{condition MNI star}) about normalizers, in this work we also
deal with the restriction
\begin{equation}
\label{condition MCI star}
|C_G(g):\langle g \rangle|\le p^k
\quad
\text{for every $\langle g \rangle \not\trianglelefteq G$.}
\end{equation}

Now we state the main results of the paper, Theorems A and B below.
Our goal is to study non-Dedekind finite $p$-groups satisfying any of the conditions
(\ref{condition MNI}), (\ref{condition MNI star}) or (\ref{condition MCI star}).
It is convenient to introduce the following notation: if $G$ is a non-Dedekind finite group, we define
\[
\mni(G)
=
\max \{ |N_G(H):H|\mid \text{$H$ is not normal in $G$} \},
\]
and its two variants,
\[
\mni^*(G)
=
\max \{ |N_G(\langle g \rangle):\langle g \rangle|
\mid \text{$\langle g \rangle$ is not normal in $G$} \},
\]
and
\[
\mci^*(G)
=
\max \{ |C_G(g):\langle g \rangle|
\mid \text{$\langle g \rangle$ is not normal in $G$} \}.
\]
This way, conditions (\ref{condition MNI}), (\ref{condition MNI star}) and
(\ref{condition MCI star}) can be rephrased as $\mni(G)\le p^k$, $\mni^*(G)\le p^k$ and
$\mci^*(G)\le p^k$, respectively.
Of course, in studying the groups satisfying any of these conditions, we may assume that
the equality holds.

Clearly, we have
\begin{equation}
\label{relation invariants}
\mci^*(G)\le \mni^*(G)\le \mni(G)
\end{equation}
for every non-Dedekind finite group $G$, and so (\ref{condition MCI star}) is the weaker of the three restrictions (\ref{condition MNI}), (\ref{condition MNI star}) and
(\ref{condition MCI star}).
As we will prove in Proposition \ref{mni=mni*}, \emph{if $G$ is a finite $p$-group\/} then actually $\mni(G)=\mni^*(G)$, and so (\ref{condition MNI star}) is not weaker than
(\ref{condition MNI}) in that case.
Thus we only need to deal with the two conditions $\mni(G)=p^k$ and $\mci^*(G)=p^k$. However, Theorems A and B are also stated in the case $\mni^*(G)=p^k$ for completeness.
It is important to stress that the equality $\mni(G)=\mni^*(G)$ does not hold in general for finite groups.
For instance, the alternating group $A_5$ is a counterexample.

In our first result we improve the aforementioned bound of X.\ Zhang and Guo for condition
(\ref{condition MNI}) and $p>2$, from a quadratic to a linear function in the exponent of
$p$.
The bound that we get is best possible, and is valid under the weaker hypothesis that
(\ref{condition MCI star}) holds.

\begin{thmA}
Let $G$ be a non-abelian finite $p$-group, where $p>2$.
If either $\mni(G)=p^k$, $\mni^*(G)=p^k$, or $\mci^*(G)=p^k$, then we have
$|G|\le p^{2k+2}$.
This bound is sharp under all three conditions.
\end{thmA}

Now we deal with the case where $p=2$, which has only been considered before in the literature under condition (\ref{condition MNI}), and only for $k=1$.
We show that the finite $2$-groups satisfying any of the conditions (\ref{condition MNI}), (\ref{condition MNI star}) or (\ref{condition MCI star}) are either of bounded order, or belong to one of two infinite families $\FF_1$ and $\FF_2$, which we describe next.

Both families consist of $2$-groups of the form $G=\langle b, A \rangle$, where $A$ is normal abelian, and $b^2\in\Omega_1(A)$.
In the family $\FF_1$, we take $A$ of exponent $2^n$ and $a^b=a^s$ for every $a\in A$, where either $s=-1$ and $n\ge 1$, or $s=-1+2^{n-1}$ and $n\ge 3$.
These groups can be constructed with the help of the theory of cyclic extensions
(see Section III.7 of \cite{zas}), and any element in $\Omega_1(A)$ is a valid choice for $b^2$.
Observe that $Z(G)=C_A(b)=\Omega_1(A)$ for $n\ge 2$, and that the only Dedekind groups in the family $\FF_1$ correspond to $A\cong C_2\times \cdots \times C_2$, or to
$A\cong C_4\times C_2\times \cdots \times C_2$ and $b^2\in A^2\smallsetminus 1$.
If $G\in\FF_1$ is not a Dedekind group and $A$ is of rank $r$, then the values of
$\mni(G)$, $\mni^*(G)$, and $\mci^*(G)$ are as follows (see Theorem \ref{values for F1}):
\[
\mni(G)
=
\mni^*(G)
=
\begin{cases}
2^r,
&
\text{if $b^2\in A^2$,}
\\
2^{r-1},
&
\text{if $b^2\not\in A^2$,}
\end{cases}
\]
and
\[
\mci^*(G)
=
\begin{cases}
2^r,
&
\text{if $G\smallsetminus A$ contains an element of order $2$,}
\\
2^{r-1},
&
\text{otherwise.}
\end{cases}
\]
On the other hand, in the family $\FF_2$, we have $A=\langle a_1 \rangle \times A^*$, where $o(a_1)=2^n$ and $A^*$ is non-trivial of order $2^m$.
The action of $b$ on $A$ is given by $a_1^b=a_1^sz$ and $(a^*)^b=(a^*)^s$ for every
$a^*\in A^*$, where $z\in\Omega_1(A^*)$, $z\ne 1$, and either $s=-1$ or $s=-1+2^{n-1}$. We assume that $n\ge 2$ if $s=-1$, and that $n\ge 3$ and $n\ge m$ if $s=-1+2^{n-1}$.
(The condition $n\ge m$ guarantees that the automorphism induced by conjugation by $b$ is of order $2$ when $s=-1+2^{n-1}$.)
Again by the theory of cyclic extensions, for given $A$ and $s$, any choice of
$z\in\Omega_1(A^*)\smallsetminus 1$ and $b^2\in\Omega_1(A)$ will define a group in
$\FF_2$.
Since $\langle a_1 \rangle\not\trianglelefteq G$, the family $\FF_2$ consists entirely of non-Dedekind groups.
As above, we have $Z(G)=C_A(b)=\Omega_1(A)$.
In this case, we have (see Theorem \ref{values for F2})
\[
\mni(G)
=
\mni^*(G)
=
\begin{cases}
2^{m+1},
&
\parbox[t]{.45\textwidth}
{if $A^*$ is elementary abelian, and $b^2\in A^2$ or $b^2z\in A^2$.}
\\[15pt]
2^m,
&
\text{otherwise,}
\end{cases}
\]
and
\[
\mci^*(G)
=
\begin{cases}
2^{m+1},
&
\parbox[t]{.55\textwidth}
{if $A^*$ is elementary abelian, and $G\smallsetminus A$ contains an element of order $2$,}
\\[15pt]
2^m,
&
\text{otherwise.}
\end{cases}
\]

\vspace{10pt}

\noindent
Thus the values of $\mni(G)$, $\mni^*(G)$, and $\mci^*(G)$ vary with the rank of $A$ in the case of family $\FF_1$, and with the order of $A^*$, in the case of $\FF_2$.
In any case, they are independent of $n$, which allows to have $2$-groups of arbitrarily large order with a fixed value of any of the three invariants we are considering.
Our second main result shows that the groups in $\FF_1$ and $\FF_2$ are the only such examples.

\begin{thmB}
Let $G$ be a non-Dedekind finite $2$-group, and suppose that either
$\mni(G)=2^k$, $\mni^*(G)=2^k$, or $\mci^*(G)=2^k$.
Then there exists a polynomial function $f(k)$ of degree four such that,
if $|G|>2^{f(k)}$, then $G$ belongs to one of the families $\FF_1$ or $\FF_2$.
\end{thmB}

Observe that the infinite families obtained by Q.\ Zhang and Gao in their classification of finite $2$-groups with $\mni(G)=2$ all belong to our family $\FF_1$, by choosing
$A\cong C_{2^{n-1}}$ in the case of $2$-groups of maximal class, and
$A\cong C_{2^{n-2}}\times C_2$ for the groups in (\ref{k=1 1st infinite family}) and
(\ref{k=1 2nd infinite family}).
On the other hand, note that the groups of the family $\FF_2$ are not present in the case $k=1$.
Indeed, according to the values of $\mni(G)$ given above, if $G$ lies in $\FF_2$ and
$\mni(G)=2$, then necessarily $A^*=\langle z \rangle$ is of order $2$, $b^2\not\in A^2$, and $b^2z\not\in A^2$.
Since also $z\not\in A^2$, it follows that the subgroup $\Omega_1(A)$, which is of order $4$, has three elements outside $A^2$.
Consequently $A^2=1$, and this implies that $n=1$, which is never the case in the family
$\FF_2$.

\vspace{10pt}

\noindent
\textit{Notation.\/}
We use standard notation in group theory.
In particular, $d(G)$ stands for the minimum number of generators
of a finitely generated group $G$.
We write $\exp G$ for the exponent of a finite group $G$.
If $G$ is a finite $p$-group and $i\ge 0$, then $\Omega_i(G)$ denotes the subgroup
generated by the elements of $G$ of order at most $p^i$, and $G^{p^i}$
is the subgroup generated by the $p^i$th powers of all elements of $G$.
We also put $\Omega_i(G)=1$ and $G^{p^i}=G$ for $i<0$.
On the other hand, we write $\Cl_G(g)$ for the conjugacy class of an element $g$ in a group $G$.

\section{Preliminary results}

In this section, we collect some preliminary results that are needed for
the proof of Theorems A and B.
First of all, we prove that $\mni(G)=\mni^*(G)$ for every non-Dedekind finite $p$-group
$G$.
We need the following lemma.

\begin{lem}
\label{normalizer and normal closure}
Let $G$ be a finite $p$-group, and let $g\in G$.
Then $|G:N_G(\langle g \rangle)|\le |\langle g \rangle^G:\langle g \rangle|$.
\end{lem}

\begin{proof}
Put $p^r=o(g)$, $p^s=|G:N_G(\langle g \rangle)|$, and
$p^t=|\langle g \rangle^G:\langle g \rangle|$.
Let us assume, by way of  contradiction, that $s\ge t+1$.
Let $n_r$ be the number of subgroups of $\langle g \rangle^G$ of order $p^r$.
Observe that $n_r$ is at least the number of conjugates of $\langle g \rangle$ in $G$, that is, $p^s$.
Hence
\begin{equation}
\label{lower bound n_r}
n_r\ge p^{t+1}.
\end{equation}
On the other hand, $n_r$ equals the number of elements of order $p^r$ in
$\langle g \rangle^G$ divided by $\varphi(p^r)$.
Since $|\langle g \rangle^G|=p^{r+t}$, it follows that
\begin{equation}
\label{upper bound n_r}
n_r \le \frac{p^{r+t}-1}{\varphi(p^r)} < \frac{p^{t+1}}{p-1}.
\end{equation}
Now by comparing (\ref{lower bound n_r}) and (\ref{upper bound n_r}) we derive a contradiction.
\end{proof}

The previous result is not always valid if $G$ is not a finite $p$-group.
For example, if $g=(1\ 2)(3\ 4)$ then
$N_{A_4}(\langle g \rangle)=\langle g \rangle^{A_4}
=\langle (1\ 2)(3\ 4),(1\ 3)(2\ 4) \rangle$, the Klein four-group.
Thus $|A_4:N_{A_4}(\langle g \rangle)|>|\langle g \rangle^{A_4}:\langle g \rangle|$ in this case.

\begin{pro}
\label{mni=mni*}
Let $G$ be a finite $p$-group, and let $H$ be a subgroup of $G$.
Then
\[
|N_G(H):H|\le |N_G(\langle h \rangle):\langle h \rangle|
\]
for every $h\in H$.
As a consequence, if $G$ is not Dedekind then $\mni(G)=\mni^*(G)$.
\end{pro}

\begin{proof}
By applying Lemma \ref{normalizer and normal closure} to the group $N_G(H)$ and the element $h$, we get
\[
|N_G(H):N_{N_G(H)}(\langle h \rangle)|
\le
| \langle h \rangle^{N_G(H)}:\langle h \rangle |.
\]
It follows that
\[
| N_G(H):N_G(\langle h \rangle) | \le | H:\langle h \rangle |,
\]
and consequently $|N_G(H):H|\le |N_G(\langle h \rangle):\langle h \rangle|$, as desired.
\end{proof}

The equality $\mni(G)=\mni^*(G)$ does not hold for all finite groups.
For example, we have $\mni(A_5)=3$ but $\mni^*(A_5)=2$.

\vspace{7pt}

Next, given a finite abelian subgroup $A$ of a group $G$, we
analyse when all subgroups of $A$ are normal in $G$.

\begin{pro}
\label{when all normal}
Let $A$ be a finite abelian subgroup of a group $G$.
Then the following are equivalent:
\begin{enumerate}
\item
All subgroups of $A$ are normal in $G$.
\item
All direct factors of $A$ are normal in $G$.
\item
For every direct product decomposition
$A=\langle a_1 \rangle \times \cdots \times \langle a_r \rangle$
with $o(a_1)=\exp A$, the subgroups
$\langle a_i \rangle$ and $\langle a_1a_j \rangle$ are normal
in $G$ for every $i=1,\ldots,r$ and $j=2,\ldots,r$.
\item
There is a direct product decomposition
$A=\langle a_1 \rangle \times \cdots \times \langle a_r \rangle$
with $o(a_1)=\exp A$,
such that the subgroups $\langle a_i \rangle$ and $\langle a_1a_j \rangle$
are normal in $G$ for every $i=1,\ldots,r$ and $j=2,\ldots,r$.
\item
For every $g\in G$, there exists an integer $s=s(g)$ such that
$a^g=a^s$ for every $a\in A$.
\end{enumerate}
If these properties are fulfilled, then $G/C_G(A)$ embeds in $\UU(\Z/e\Z)$,
where $e$ is the exponent of $A$.
\end{pro}

\begin{proof}
It is clear that (ii) follows from (i), that (iv) follows from (iii),
and that (i) is a consequence of (v).
We complete the equivalence of the five conditions by showing that (ii)
implies (iii), and that (iv) implies (v).

Let $A=\langle a_1 \rangle \times \cdots \times \langle a_r \rangle$
be a decomposition with $o(a_1)=\exp A$.
Then also $o(a_1a_j)=\exp A$ for every $j=2,\ldots,r$, and $\langle a_1a_j \rangle$ is a direct factor of $A$ by \cite[2.1.2]{kur-ste}.
Thus (iii) follows from (ii).

Let now $A=\langle a_1 \rangle \times \cdots \times \langle a_r \rangle$
be a direct decomposition of $A$ which fulfils the conditions in (iv).
Then for every $g\in G$ and every $i=1,\ldots,r$ and $j=2,\ldots,r$, there exist integers
$s_i$, $t_j$ such that $a_i^g=a_i^{s_i}$ and $(a_1a_j)^g=(a_1a_j)^{t_j}$.
Consequently
\[
a_1^{t_j}a_j^{t_j} = (a_1a_j)^g = a_1^g a_j^g = a_1^{s_1} a_j^{s_j},
\]
and so $t_j\equiv s_1 \pmod{o(a_1)}$, and $t_j\equiv s_j \pmod{o(a_j)}$ for every $j=2,\ldots,r$.
Since $o(a_1)=\exp A$, it follows that $o(a_j)$ divides $o(a_1)$,
and consequently $s_j\equiv s_1 \pmod{o(a_j)}$.
Thus $a_j^g = a_j^{s_j} = a_j^{s_1}$, and (v) holds with $s=s_1$.

Finally, observe that the last assertion of the theorem follows immediately from (v).
\end{proof}

If $G$ is a non-Dedekind finite group, we denote by $R(G)$ the intersection
of all non-normal subgroups of $G$.
Note that $R(G)$ coincides with the intersection of all non-normal
\emph{cyclic\/} subgroups of $G$.
We will need the following result of Blackburn \cite[Theorem 1]{bla}:
if $G$ is a non-Dedekind finite $p$-group and $R(G)\ne 1$, then $p=2$,
$R(G)$ is of order $2$, and $G$ belongs to one of the following types:
\begin{enumerate}
\item[(R1)]
Isomorphic to $Q_8\times C_4\times E$, where $E$ is elementary abelian.
\item[(R2)]
Isomorphic to $Q_8\times Q_8\times E$, where $E$ is elementary abelian.
\item[(R3)]
A $Q$-group which is not a Dedekind group.
\end{enumerate}
Here, a $Q$-group is a group $G=\langle A, b \rangle$, where $A$ is abelian
but not elementary abelian, $a^b=a^{-1}$ for all $a\in A$, and $b^2\in A$
is of order $2$.
Note that all Hamiltonian groups (i.e.\ non-abelian Dedekind groups) are
$Q$-groups, as well as all generalised quaternion groups.
On the other hand, $Q$-groups belong to the family $\FF_1$ that we have defined in the introduction.

\vspace{10pt}

The following remark will be useful in the proof of our results.

\begin{lem}
\label{cyclic avoiding R(G)}
Let $G$ be a non-Dedekind finite $p$-group, and let $C$ be a cyclic subgroup of $G$.
Then there exists a non-normal cyclic subgroup $C^*$ of $G$ such that
$|C\cap C^*|\le |R(G)|$.
\end{lem}

\begin{proof}
Let $C^*$ be a non-normal cyclic subgroup of $G$ for which the intersection $C\cap C^*$ has minimum order, and assume by way of contradiction that $|C\cap C^*|>|R(G)|$.
Then, by the definition of $R(G)$, there exists a non-normal cyclic subgroup $D$ of $G$ such that $C\cap C^*\not\le D$.
Since $C$ is a cyclic finite $p$-group, we have either $C\cap C^*\le C\cap D$ or
$C\cap D<C\cap C^*$.
Now the former case is impossible, since $C\cap C^*\not\le D$, and the latter is contrary to the choice of $C^*$.
This contradiction proves the claim.
\end{proof}

If $G$ is a non-Dedekind finite $p$-group, then $\mni(G)$ and $\mni^*(G)$ are greater than $1$, since $G$ satisfies the normalizer condition.
It may happen however that $\mci^*(G)=1$, but only in very few cases, as we next show.

\begin{lem}
\label{mci^*(G)=1}
Let $G$ be a non-Dedekind finite $p$-group.
Then $\mci^*(G)=1$ if and only if $p=2$ and $G\cong Q_{2^n}$ is a generalised quaternion group, with $n\ge 4$.
\end{lem}

\begin{proof}
Assume first that $\mci^*(G)=1$.
If $\langle g \rangle$ is not normal in $G$, then $C_G(g)=\langle g \rangle$,
and consequently $Z(G)\le \langle g \rangle$.
Thus $Z(G)\le R(G)$.
Since $G$ is a finite $p$-group, it follows that $Z(G)=R(G)$ is of order $2$,
and $G$ is isomorphic to one of the groups given in (R1), (R2), and (R3).
Then $G$ is necessarily a generalised quaternion group, since otherwise
$|Z(G)|\ge 4$.
The converse can be easily checked.
\end{proof}

The following result is a particular case of a theorem of Kummer about the $p$-adic valuation of a binomial coefficient \cite[Theorem 10.2.2]{and-and}.

\begin{lem}
\label{p-adic valuation}
Let $p$ be a prime, and let $m\in\N$.
Then for every $1\le i\le p^m$, if $p^{\ell}$ is the highest power of $p$ which divides $i$, the binomial coefficient $\binom{p^m}{i}$ is divisible by $p^{m-\ell}$.
As a consequence, if $p>2$ and $2\le i\le m+1$ then $\binom{p^m}{i}$ is divisible by $p^{m-i+2}$.
\end{lem}

If $y$ is an element of a group $G$ such that the normal closure $\langle y \rangle^G$ is abelian, then we have
\begin{equation}
\label{power of product}
(xy)^n = x^n y^n [y,x]^{\binom{n}{2}} [y,x,x]^{\binom{n}{3}} \ldots
[y,x,\overset{n-1}{\ldots},x]^{\binom{n}{n}},
\end{equation}
for every $x\in G$ and for every $n\in\N$.
Similarly, if the derived subgroup of $\langle x,y \rangle$ is abelian, then
\begin{equation}
\label{commutator with power}
[x^n,y]  = [x,y]^n [x,y,x]^{\binom{n}{2}} [x,y,x,x]^{\binom{n}{3}} \ldots
[x,y,x,\overset{n-1}{\ldots},x]^{\binom{n}{n}}.
\end{equation}

The following lemma is well-known to experts (it can be used, for example, to show that certain metacyclic $p$-groups are split).
However, since we have not found a clear reference in the literature, we have decided to include it, for the convenience of the reader, in the precise form that we are going to need it.
Recall that, if $G=\langle g \rangle$ is a finite $p$-group and $p>2$, then the (only) Sylow $p$-subgroup of $\Aut G$ is cyclic, generated by the automorphism $g\mapsto g^{1+p}$.

\begin{lem}
\label{power=1}
Let $G$ be a finite $p$-group, where $p$ is an odd prime, and let $K$ be a normal cyclic subgroup of $G$, of order $p^s$.
If $g^{p^t}\in K^{p^t}$ for some positive integer $t\le s$, then there exists
$h\in gK$ such that  $h^{p^t}=1$ and $\langle h \rangle \cap K=1$.
\end{lem}

\begin{proof}
Let $p^m$ be the order of $g$ modulo $K$.
Then $m\le t$, and
\[
| \langle g^{p^m} \rangle : K^{p^t} |
\le
| \langle g^{p^m} \rangle : \langle g^{p^t} \rangle |
\le
p^{t-m}
=
|K^{p^m}:K^{p^t}|.
\]
(Note that we need the condition $t\le s$ for the last equality to hold.)
Hence $| \langle g^{p^m} \rangle | \le |K^{p^m}|$.
Since $g^{p^m}\in K$ and $K$ is cyclic, it follows that $g^{p^m}\in K^{p^m}$.
Let $y\in K$ be such that $g^{p^m}=y^{p^m}$, and put $h=gy^{-1}$.
Let also $p^n$ be the order of $y$.

Since $\langle y \rangle$ is normal in $G$, we can use (\ref{power of product}) and get
\begin{multline}
\label{power of h}
h^{p^m}
=
g^{p^m} y^{-p^m} [y^{-1},g]^{\binom{p^m}{2}}
[y^{-1},g,g]^{\binom{p^m}{3}}
\ldots
[y^{-1},g,\overset{i-1}{\ldots},g]^{\binom{p^m}{i}}
\\
\ldots
[y^{-1},g,\overset{p^m-1}{\ldots},g]^{\binom{p^m}{p^m}}.
\end{multline}
Now, observe that $g^{p^m}$ acts as the identity on $\langle y \rangle$, so
$\langle g \rangle$ embeds as a subgroup of order at most $p^m$ in
$\Aut \langle y \rangle$.
Since $p>2$, the only subgroup of order $p^m$ in $\Aut \langle y \rangle$ is generated by the automorphism $y\longmapsto y^{1+p^{n-m}}$.
This implies that $[y^{-1},g]\in \langle y^{p^{n-m}} \rangle =\Omega_m(\langle y \rangle)$.
Since this subgroup is normal in $G$, it follows that
\[
[y^{-1},g,\overset{i-1}{\ldots},g] \in \Omega_{m-i+2}(\langle y \rangle)
\]
for every $i\ge 2$, and so
\begin{equation}
\label{eq1}
[y^{-1},g,\overset{i-1}{\ldots},g] ^{p^{m-i+2}} = 1
\quad
\text{for $2\le i\le m+1$,}
\end{equation}
and
\begin{equation}
\label{eq2}
[y^{-1},g,\overset{i-1}{\ldots},g] = 1
\quad
\text{for $i>m+1$.}
\end{equation}
Since, according to Lemma \ref{p-adic valuation}, the binomial coefficient
$\binom{p^m}{i}$ is divisible by $p^{m-i+2}$ for $2\le i\le m+1$, it follows from
(\ref{power of h}), (\ref{eq1}) and (\ref{eq2}) that $h^{p^m} = g^{p^m} y^{-p^m} = 1$.
Thus also $h^{p^t}=1$.

Finally, observe that $o(hK)=o(gK)=p^m$ in the quotient group $G/K$.
Since $h^{p^m}=1$, this implies that $\langle h \rangle \cap K=1$, and we are done.
\end{proof}

\begin{rmk}
The restriction $t\le s$ in the statement of the previous lemma is needed to avoid artificial counterexamples.
For instance, if $G=\langle g \rangle$ is cyclic of order $p^t$ and $K$ is a non-trivial proper subgroup of $G$, then $g^{p^t}\in K^{p^t}$, but it is impossible to find an element $h$ as in
Lemma \ref{power=1}.
\end{rmk}

\section{Odd primes}

In this section we prove Theorem A.
Since $\mci^*(G)\le \mni^*(G)=\mni(G)$, it suffices to prove the result when
$\mci^*(G)=p^k$.
We begin with a particular case.

\begin{pro}
\label{all normal}
Let $p$ be an odd prime, and let $G$ be a non-abelian finite $p$-group such that
$\mci^*(G)=p^k$.
If $G$ possesses a maximal abelian normal subgroup all of whose subgroups are normal in $G$, then
$|G|\le p^ {2k+2}$.
\end{pro}

\begin{proof}
Let $A$ be a maximal abelian normal subgroup all of whose subgroups are normal in $G$, and let $p^n$ be the exponent of $A$.
By Proposition \ref{when all normal}, the quotient group $G/A$ embeds in $\UU(\Z/p^n\Z)$, which is cyclic of order $p^{n-1}(p-1)$.
Thus $G/A$ is cyclic of order $p^t$ for some $t\le n-1$.
This implies in particular that $n\ge 2$, since $G$ is non-abelian.
Since $p>2$, we can choose a generator $gA$ of $G/A$ such that $a^g=a^{1+p^{n-t}}$ for every
$a\in A$.
As a consequence, $C_A(g)=\Omega_{n-t}(A)$ and $C_A(g^{p^{t-1}})=\Omega_{n-1}(A)$.
In particular, we have $g^{p^t}\in\Omega_{n-t}(A)$.

Now write $A=B\times C$, where $B$ is homocyclic of exponent $p^n$, and $\exp C\le p^{n-1}$.
Thus $|B|=p^{rn}$ for some $r\ge 1$.
Also, we have $\Omega_{n-t}(A)=B^{p^t}\times \Omega_{n-t}(C)$, and so
$g^{p^t}=b^{p^t}c$, with $b\in B$ and $c\in \Omega_{n-t}(C)$.
By applying Lemma \ref{power=1} to the quotient $G/\Omega_{n-t}(C)$, with
$\langle b \rangle \Omega_{n-t}(C)/\Omega_{n-t}(C)$ playing the role of $K$, there exists
$h\in gA$ such that $h^{p^t}\in \Omega_{n-t}(C)$.
Note that $h$ plays the same role as $g$, in the sense that $hA$ is a generator of
$G/A$, and $C_A(h^{p^{t-1}})=\Omega_{n-1}(A)$.

Since $o(hA)=p^t$, we have $A\cap \langle h \rangle=\langle h^{p^t} \rangle \le C$.
Hence $B\cap \langle h \rangle=1$.
Since $[h^{p^{t-1}},B]=B^{p^{n-1}}\ne 1$, it follows that $\langle h^{p^{t-1}} \rangle$ is not normal in $G$.
By the condition $\mci^*(G)=p^k$, we have
\begin{equation}
\label{upper bound}
| C_G(h^{p^{t-1}}):\langle h^{p^{t-1}} \rangle | \le p^k.
\end{equation}
Also,
\begin{equation}
\label{lower bound}
\begin{aligned}
| C_G(h^{p^{t-1}}):\langle h^{p^{t-1}} \rangle |
& \ge
| \langle h \rangle \Omega_{n-1}(A):\langle h^{p^{t-1}} \rangle |
\\
& =
| \langle h \rangle:\langle h^{p^{t-1}} \rangle |
\,
|  \Omega_{n-1}(A): \Omega_{n-1}(A) \cap \langle h \rangle |
\\
&=
p^{t-1}
\,
|  \Omega_{n-1}(A): \Omega_{n-1}(A) \cap \langle h \rangle |.
\end{aligned}
\end{equation}
Since $\Omega_{n-1}(A)=B^p \times C$ and $A \cap \langle h \rangle \le C$, we have
\begin{equation}
\label{index}
|  \Omega_{n-1}(A): \Omega_{n-1}(A) \cap \langle h \rangle |
=
|  B^p | \, | C : C \cap \langle h \rangle |.
\end{equation}
This, together with (\ref{upper bound}) and (\ref{lower bound}), yields
\[
| B^p | \le p^{k-t+1}.
\]
Since $|B^p|=p^{r(n-1)}$, it follows that
\[
k-t+1 \ge r(n-1) \ge r+n-2,
\]
by using that both $r$ and $n-1$ are greater than or equal to $1$.
Hence
\begin{equation}
\label{bound r+n}
r+n \le k-t+3.
\end{equation}
On the other hand, since $|  \Omega_{n-1}(A): \Omega_{n-1}(A) \cap \langle h \rangle | \ge
| \Omega_{n-1}(A) | / p^{n-1}$, again by (\ref{upper bound}) and (\ref{lower bound}) we get
\begin{equation}
\label{order omega}
|\Omega_{n-1}(A)| \le p^{k-t+n}.
\end{equation}

Now, since $A$ is abelian, we have $|A| = |A^{p^{n-1}}| \, |\Omega_{n-1}(A)| $.
Hence
\[
|A| = |B^{p^{n-1}}| \, |\Omega_{n-1}(A)| = p^r \,  |\Omega_{n-1}(A)|,
\]
and so, by (\ref{bound r+n}) and (\ref{order omega}),
\[
|A| \le p^{2k-2t+3} \le p^{2k-t+2}.
\]
It follows that
\[
|G| = |G/A| \, |A| = p^t \, |A| \le p^{2k+2},
\]
which completes the proof.
\end{proof}

In order to attack the general case, we need the following result about automorphisms of an abelian group of the form $C_{p^n}\times C_p\times \cdots \times C_p$, where
$n\ge 2$.

\begin{lem}
\label{Aut B}
Let $p$ be an odd prime, and let $B=\langle b_1 \rangle \times B^*$, where
$o(b_1)=p^n\ge p^2$, and $B^*$ is elementary abelian of order $p^m$.
If $Q$ is the subgroup of $\Aut B$ formed by the $p$-automorphisms that act as the identity on $B^*$, then we have $Q = \langle \varphi_1 \rangle \times Q^*$, where
$\varphi_1$ is the automorphism of $B$ defined by the rules
\[
\varphi_1(b_1)=b_1^{1+p},
\quad
\varphi_1(b^*)=b^*,
\
\text{for every $b^*\in B^*$,}
\]
and $Q^*$ is the subgroup of $Q$ formed by the automorphisms which also act as the identity on $B/B^*$.
Also, we have $Q^*\cong B^*$, and thus $Q\cong C_{p^{n-1}}\times C_{p^m}$.
\end{lem}

\begin{proof}
Given $\psi\in Q$, let us write $\psi(b_1)=b_1^i b^*$ with $i\in\Z$ and $b^*\in B^*$.
Observe that $i$ is not divisible by $p$, since $o(\psi(b_1))=p^n$ and $n\ge 2$.
Since $\psi$ is a $p$-automorphism and $p$ is odd, $i$ must be a power of $1+p$ modulo
$p^n$.
Then we have $\psi=\varphi\varphi^*=\varphi^*\varphi$, where $\varphi$ and $\varphi^*$ are the automorphisms in $Q$ sending $b_1$ to $b_1^i$ and to $b_1b^*$, respectively.
Since $\varphi$ is a power of $\varphi_1$, it follows that
$Q=\langle \varphi_1 \rangle \times Q^*$.
Finally, since every $b^*\in B^*$ induces the automorphism $b_1\mapsto b_1b^*$, we have $Q^*\cong B^*$.
\end{proof}

Now we complete the proof of Theorem A.
Recall that, for $p$ an odd prime, a finite $p$-group $G$ is called \emph{$p$-central\/} if
$\Omega_1(G)\le Z(G)$.

\begin{thm}
Let $p$ be an odd prime, and let $G$ be a non-abelian finite $p$-group such that
$\mci^*(G)=p^k$.
Then $|G|\le p^ {2k+2}$.
\end{thm}

\begin{proof}
According to Proposition \ref{all normal}, we may assume that there exists at least a subgroup $H$ of $G$ satisfying the following condition (C): $H$ is an abelian normal subgroup of $G$ containing a subgroup which is not normal in $G$.

Among all subgroups of $G$ satisfying (C), we choose one, $B$, which is minimal in the following sense:
\begin{enumerate}
\item
If $H$ satisfies (C) then $\exp B\le \exp H$.
\item
If $H$ satisfies (C) and $\exp B=\exp H$, then $|B|\le |H|$.
\end{enumerate}

Let $p^n$ be the exponent of $B$, and let $b_1\in B$ be such that
$\langle b_1 \rangle \not\trianglelefteq G$.
Since $\mci^*(G)=p^k$, we have
\begin{equation}
\label{centralizer b1}
|C_G(b_1)| = |C_G(b_1):\langle b_1 \rangle| \, |\langle b_1 \rangle| \le p^{k+n}.
\end{equation}
If $n=1$ then $|C_G(b_1)|\le p^{k+1}$, and in particular $|B|\le p^{k+1}$.
Since $|G:C_G(b_1)|=|\Cl_G(b_1)|\le |B|$, it follows that $|G|\le p^{2k+2}$, and we are done.
Hence we assume that $n\ge 2$ in the remainder of the proof.

\vspace{8pt}

\noindent
\textit{Claim 1}.
$G$ is a $p$-central group.

\vspace{8pt}

Let $W$ be a maximal elementary abelian normal subgroup of $G$.
Since $\exp B\ge p^2$, it follows from the choice of $B$ that
$\langle w \rangle \trianglelefteq G$ for every $w\in W$.
Hence $W\le Z(G)$.
On the other hand, by a well-known theorem of Alperin
\cite[Chapter III, Theorem 12.1]{hup}, we have $\Omega_1(C_G(W))=W$, since $p>2$.
Hence $\Omega_1(G)=W$ and, in particular, $\Omega_1(G)\le Z(G)$.
Thus $G$ is $p$-central, as claimed.

\vspace{8pt}

Let us continue analysing the structure of $G$.
Since $\exp B^p<\exp B$, we have $\langle b_1^p \rangle \trianglelefteq G$ by the choice of $B$.
Then
\[
[b_1,g]^p = [b_1^p,g] \in \langle b_1^p \rangle,
\]
for every $g\in G$, and consequently $[b_1,g]\in \langle b_1 \rangle \Omega_1(B)$.
Hence $\langle b_1 \rangle \Omega_1(B)$ is normal in $G$, and again by the minimality of $B$, we have $B=\langle b_1 \rangle \Omega_1(B)$.
Thus $B=B_1 \times B^*$, where $B_1=\langle b_1 \rangle$, and $B^*$ is elementary abelian (and so central in $G$).
It follows that $C_G(B)=C_G(b_1)$, and so
\begin{equation}
\label{bound C_G(B)}
|C_G(B)| \le p^{k+n},
\end{equation}
by (\ref{centralizer b1}).
We write $C$ for $C_G(B)$ in the remainder of the proof.

It also follows from the previous paragraph that $o(b_1)=p^n$.
If we put $|B^*|=p^{r-1}$ then $B$ is as in the statement of Lemma \ref{Aut B}, with $r-1$ playing the role of $m$.
Also,
\begin{equation}
\label{order B}
|B| = p^{n+r-1}.
\end{equation}
Let $Q$ be the subgroup of $\Aut B$ which was defined in Lemma \ref{Aut B}.
Since every $g\in G$ induces a $p$-automorphism of $B$ which acts as the identity on $B^*$, there is an embedding $\Phi \colon G/C \longrightarrow Q$.
By Lemma \ref{Aut B}, we have
\[
G/C
\cong C_{p^t}\times C_p\times \overset{s-1}{\cdots} \times C_p
\]
for some $t\le n-1$ and $s\le r$.
In particular, $G/C$ is abelian, $\exp G/C=p^t$, and
\begin{equation}
\label{order G/C}
|G:C| = p^{t+s-1}.
\end{equation}
In the following, we let $g_1$ denote an element whose image $\overline{g_1}$ in $G/C$ has order $p^t$.
Then $\overline{g_1}$ need not correspond to a power of $\varphi_1$ under $\Phi$, but if
$t\ge 2$ then $\Phi(\overline{g_1})=\varphi_1^i\varphi^*$ for some $\varphi^*\in Q^*$, and for some $i$ which is divisible by $p^{n-t-1}$ but not by $p^{n-t}$.
Now $\varphi_1^i$ generates the same subgroup of $Q$ as the automorphism sending $b_1$ to $b_1^{1+p^{n-t}}$, and $\varphi^*$ acts trivially on $B^p$.
Hence we get
\begin{equation}
\label{non-trivial comm in B_1}
[b_1,g_1^{p^{t-1}}], [b_1^{p^{t-1}},g_1] \in B_1^{p^{n-1}} \smallsetminus 1,
\ \
[b_1^{p^{t-1}},g_1^p]=1,
\quad
\text{for $t\ge 2$.}
\end{equation}
We will need these facts later on.
It also follows that
\begin{equation}
\label{cent g^{p^{t-1}}}
C_G(g_1^{p^{t-1}})=B^p\Omega_1(B),
\end{equation}
an equality that holds for every value of $t\ge 1$.

\vspace{8pt}

\noindent
\textit{Claim 2}.
$|\Omega_1(G)|\ge p^{s+1}$ and $\Omega_1(G)\le C$.

\vspace{8pt}

Since $G$ is $p$-central, we have $|\Omega_1(G)|\ge |G:G^p|$ by
\cite[Theorem C]{gon-wei}.
Now we consider two cases, according as $C\le G^p$ or not.
If $C\le G^p$ then, since $G/C$ is abelian, it follows in particular that $G'\le G^p$, and $G$ is a powerful $p$-group.
Since $b_1\in G^p$ in this case, we may write $b_1=g^{p^i}$ for some
$g\in G\smallsetminus G^p$.
But then $g\in C_G(b_1)=C\le G^p$, which is a contradiction.
Thus we have $|C:C\cap G^p|\ge p$, and
\[
|G:G^p| = |G:G^pC| \, |G^pC:G^p| = |G/C:(G/C)^p| \, |C:C\cap G^p| \ge p^{s+1}.
\]
We conclude that $|\Omega_1(G)|\ge p^{s+1}$.
Observe also that $\Omega_1(G)\le C$, since $\Omega_1(G)$ is contained in $Z(G)$.

\vspace{8pt}

\noindent
\textit{Claim 3}.
The theorem is proved if there exists $h\in C$ such that $\langle h \rangle \cap B_1=1$,
$\langle h \rangle \not\trianglelefteq G$, and $[h,y]=1$ for some $y\in G$ whose order modulo $C$ is $p^{t-1}$.

\vspace{8pt}

For such an element $h$, we have
$\langle h \rangle \cap B^p=\langle h \rangle \cap B_1^p=1$,
and consequently the order of $\langle h \rangle \cap B$ is at most $p$.
Hence
\[
|C_C(h):\langle h \rangle| \ge |B:\langle h \rangle \cap B| \ge |B|/p.
\]
Since $\mci^*(G)=p^k$, we have
\[
p^k \ge |C_G(h):\langle h \rangle| = |C_G(h):C_C(h)| \, |C_C(h):\langle h \rangle|
\ge |C_G(h)C:C| \, |B|/p.
\]
Now, since
\[
|C_G(h)C:C| \ge |\langle y \rangle C:C| \ge p^{t-1},
\]
it follows that
\[
|B| \le p^{k-t+2}.
\]
By (\ref{order B}), we have $n+r-1\le k-t+2$.
Then we conclude from (\ref{bound C_G(B)}) and (\ref{order G/C}) that
\[
|G| = |G:C| \, |C| \le p^{t+r-1} \, p^{k+n} \le p^{2k+2},
\]
as desired.

\vspace{8pt}

Observe that, in the case that $t=1$, the existence of the element $y$ in Claim 3 is straightforward, by taking $y=1$.
So in that case we only have to worry about finding an appropriate $h\in C$.

\vspace{8pt}

\noindent
\textit{Claim 4}.
We may assume that there exists en element $g\in G$ of order $p^t$ modulo $C$, such that $\langle g \rangle \cap B_1=1$ and $\langle g^{p^{t-1}} \rangle \not\trianglelefteq G$.

\vspace{8pt}

For this purpose, we consider separately the cases $t=1$ and $t\ge 2$.
Suppose first that $t=1$.
As shown above, the theorem holds if there exists $h\in C$ such that
$\langle h \rangle \cap B_1=1$ and $\langle h \rangle \not\trianglelefteq G$.
Thus we may assume that $\langle h \rangle \cap B_1 \ne 1$ whenever $h\in C$ and
$\langle h \rangle \not\trianglelefteq G$.
Now, let us choose a subgroup $J$ of $G$ such that $|J:C|=p$.
Then $J$ is not abelian, since $B\le J$ and $J\not\le C$.
Since $p>2$, it follows that $R(J)=1$ by \cite[Theorem 1]{bla}, as already mentioned.
By Lemma \ref{cyclic avoiding R(G)}, there exists $g\in J$ such that
$\langle g \rangle \cap B_1=1$ and $\langle g \rangle \not\trianglelefteq G$.
Since $g$ cannot belong to $C$, it follows that $o(\overline g)=p$ in $G/C$, as desired.

Now we deal with the case that $t\ge 2$.
It suffices to find an element $g\in g_1C$ such that $\langle g\rangle \cap B_1=1$.
Indeed, in that case we have $o(\overline g)=o(\overline g_1)=p^t$ in $G/C$, and also
$\langle g^{p^{t-1}} \rangle \not\trianglelefteq G$, since
$[b_1,g^{p^{t-1}}]=[b_1,g_1^{p^{t-1}}] \in B_1 \smallsetminus 1$ by
(\ref{non-trivial comm in B_1}).

Hence we are done if the intersection $D=\langle g_1 \rangle \cap  B_1$ is trivial, and so we assume that $D\ne 1$.
Let $p^{\ell}$ and $p^m$ be the orders of $b_1$ and $g_1$ modulo $D$.
Observe that $\ell\ge t$, since $b_1^{p^{\ell}}$ commutes with $g_1$ and
$[b_1^{p^{t-1}},g_1]\ne 1$ by (\ref{non-trivial comm in B_1}).
We also have $m\ge t$, since $g_1^{p^{t-1}}\not\in C$ and $D\subseteq C$.
Suppose first that $\ell\ge m$, so that $g_1^{p^m}\in B_1^{p^m}$.
By applying Lemma \ref{power=1} in $G/B^*$, with $B/B^*$ playing the role of $K$, it follows that there exists $g\in g_1B$ such that  $\langle g \rangle \cap B\subseteq B^*$.
Consequently $\langle g \rangle \cap B_1=1$, we are done in this case.
Assume now that $\ell<m$.
Put $y_1=b_1^{p^{t-1}}$ and $x_1=g_1^{p^{m-\ell+t-1}}$.
Then both $y_1$ and $x_1$ have order $p^{\ell-t+1}$ modulo $D$, and $x_1\in C$, since $m-\ell+t-1\ge t$.
Then there exists $h\in y_1\langle x_1 \rangle$ such that $\langle h \rangle \cap B_1=1$, either by Lemma \ref{power=1}, or even simpler, because $\langle x_1,y_1 \rangle$ is abelian.
Observe that $\langle h \rangle$ is not normal in $G$, since
\[
[h,g_1]=[b_1^{p^{t-1}},g_1]\in B_1^{p^{n-1}}\smallsetminus 1
\]
by (\ref{non-trivial comm in B_1}).
On the other hand, we have $h\in C$ and $[h,g_1^p]=[b_1^{p^{t-1}},g_1^p]=1$, again
by (\ref{non-trivial comm in B_1}).
Thus $h$ fulfills all conditions of Claim 3, which imply that $|G|\le p^{2k+2}$.
This proves Claim 4 for $t\ge 2$.

\vspace{8pt}

Finally, we use the element $g$ of Claim 4 in order to complete the proof of the theorem.
First of all, since $\langle  g^{p^{t-1}} \rangle \not\trianglelefteq G$, we have
\begin{equation}
\label{good decomposition}
\begin{split}
p^k
&\ge
|C_G(g^{p^{t-1}}):\langle g^{p^{t-1}} \rangle|
\\
&=
|C_G(g^{p^{t-1}}):\langle g^{p^{t-1}} \rangle C_C(g^{p^{t-1}}) | \,
|\langle g^{p^{t-1}} \rangle C_C(g^{p^{t-1}}):\langle g^{p^{t-1}} \rangle|
\end{split}
\end{equation}
Observe that
\begin{equation}
\label{first factor}
|C_G(g^{p^{t-1}}):\langle g^{p^{t-1}} \rangle C_C(g^{p^{t-1}}) |
\ge
|\langle g \rangle C:\langle g^{p^{t-1}} \rangle C|
=
p^{t-1},
\end{equation}
and that
\begin{equation}
\label{second factor 1}
\begin{split}
|\langle g^{p^{t-1}} \rangle C_C(g^{p^{t-1}}):
&\langle g^{p^{t-1}} \rangle|
=
|C_C(g^{p^{t-1}}):C_C(g^{p^{t-1}}) \cap \langle g^{p^{t-1}} \rangle|
\\
&\ge
|C_{B\Omega_1(G)}(g^{p^{t-1}}):C_{B\Omega_1(G)}(g^{p^{t-1}}) \cap
\langle g^{p^{t-1}} \rangle|,
\end{split}
\end{equation}
since $B\Omega_1(G)\le C$.
Now, since $\Omega_1(G)\le Z(G)$, we have
\[
C_{B\Omega_1(G)}(g^{p^{t-1}})
=
C_B(g^{p^{t-1}}) \Omega_1(G)
=
B^p\Omega_1(G),
\]
by using (\ref{cent g^{p^{t-1}}}).
(Recall that the element $g_1$ in (\ref{cent g^{p^{t-1}}}) is an arbitrary element whose image in $G/C$ has order $p^t$.)
Hence
\begin{equation}
\label{second factor 2}
|C_{B\Omega_1(G)}(g^{p^{t-1}})| \ge p^{n+s-1},
\end{equation}
by using that $|\Omega_1(G)|\ge p^{s+1}$, as proved in Claim 2.
On the other hand,
\[
C_{B\Omega_1(G)}(g^{p^{t-1}}) \cap \langle g^{p^{t-1}} \rangle
=
B^p\Omega_1(G) \cap \langle g^{p^{t-1}} \rangle
\]
is a subgroup of $\Omega_1(G)$, since
\[
(B\Omega_1(G) \cap \langle g^{p^{t-1}} \rangle)^p
\subseteq
(B\Omega_1(G))^p \cap \langle g^{p^t} \rangle
\subseteq
B_1 \cap \langle g \rangle
=
1.
\]
Thus
\[
|C_{B\Omega_1(G)}(g^{p^{t-1}}) \cap \langle g^{p^{t-1}} \rangle|
\le
p,
\]
and consequently, by (\ref{second factor 1}) and (\ref{second factor 2}),
\[
|\langle g^{p^{t-1}} \rangle C_C(g^{p^{t-1}}):\langle g^{p^{t-1}} \rangle|
\ge
p^{n+s-2}.
\]
It then follows from (\ref{good decomposition}) and (\ref{first factor}) that
$k\ge n+s+t-3$.
Hence
\[
|G| = |G:C| \, |C| \le p^{s+t-1} \, p^{k+n} = p^{k+n+s+t-1} \le p^{2k+2},
\]
as desired.
\end{proof}

The following example shows that the bound $|G|\le p^{2k+2}$ in Theorem A is best possible.

\begin{exa}
\label{bound best possible}
Let $p$ be an arbitrary prime, and let $k$ be a positive integer.
Consider the group $G$ given by the following presentation:
\[
G = \langle a,b \mid a^{p^{k+1}}=b^{\,p^{k+1}}=1,\ a^b=a^{1+p^k} \rangle.
\]
Then $Z(G)=\langle a^p,b^{\,p} \rangle$ and $o(g)=p^{k+1}$ for every
$g\in G\smallsetminus Z(G)$.
By using these two facts, one can readily check that $\mni(G)=\mni^*(G)=\mci^*(G)=p^k$.
\end{exa}

\section{The even prime}

In this section we study finite $2$-groups with a given value of $\mni(G)$, $\mni^*(G)$, or
$\mci^*(G)$.
As indicated in the introduction, in this case one cannot bound the order of the group $G$, and this is due to the existence of two infinite families $\FF_1$ and $\FF_2$ in which the group order can grow arbitrarily while $\mni(G)$, $\mni^*(G)$, and $\mci^*(G)$ remain bounded.
We begin by calculating the values of these invariants for the groups in $\FF_1$ and
$\FF_2$.
We need the following straightforward lemma.

\begin{lem}
\label{g^2}
Let $G$ be a finite $2$-group in one of the families $\FF_1$ or $\FF_2$, and let
$g\in G\smallsetminus A$.
Then:
\begin{enumerate}
\item
If $G$ lies in $\FF_1$, then $g^2=b^2$ if $s=-1$, and
$g^2\equiv b^2 \pmod{A^{2^{n-1}}}$ if $s=-1+2^{n-1}$.
\item
If $G$ lies in $\FF_2$, then $g^2=b^2$ or $b^2z$ if $s=-1$, and
$g^2\equiv b^2$ or $b^2z \pmod{A^{2^{n-1}}}$ if $s=-1+2^{n-1}$.
Both possibilities $b^2$ and $b^2z$ always occur.
\end{enumerate}
In every case, we have $g^2\in\Omega_1(A)$, and so $o(g)=2$ or $4$.
\end{lem}

\begin{thm}
\label{values for F1}
Let $G$ be a non-Dedekind $2$-group in the family $\FF_1$.
If $A$ is of rank $r$, then
\begin{equation}
\label{mci* for F1}
\mci^*(G)
=
\begin{cases}
2^r,
&
\text{if $G\smallsetminus A$ contains an element of order $2$,}
\\
2^{r-1},
&
\text{otherwise,}
\end{cases}
\end{equation}
and
\begin{equation}
\label{mni* for F1}
\mni(G)
=
\mni^*(G)
=
\begin{cases}
2^r,
&
\text{if $b^2\in A^2$,}
\\
2^{r-1},
&
\text{if $b^2\not\in A^2$.}
\end{cases}
\end{equation}
\end{thm}

\begin{proof}
First of all, observe that all subgroups of $A$ are normal in $G$.
On the other hand, we claim that $\langle g \rangle \not\trianglelefteq G$ for every
$g\in G\smallsetminus A$.
Otherwise, $[g,A]$ is contained in $\langle g^2 \rangle$, which is either trivial or of order $2$ by Lemma \ref{g^2}.
Since $a^g=a^{-1}$ or $a^g=a^{-1+2^{n-1}}$ for every $a\in A$, it follows that
either $A\cong C_2\times \cdots \times C_2$, or
$A\cong C_4\times C_2\times \cdots \times C_2$ and $g^2\in A^2\smallsetminus 1$.
In any case, $G$ is a Dedekind group, which is a contradiction.

We begin by calculating $\mci^*(G)$.
For every $g\in G\smallsetminus A$, we have
\[
C_G(g) = \langle g \rangle C_A(g) = \langle g \rangle C_A(b) =
\langle g \rangle \Omega_1(A),
\]
and so
\begin{equation}
\label{|C_G(g):<g>| for F1}
|C_G(g):\langle g \rangle| = |\Omega_1(A):\Omega_1(A)\cap \langle g \rangle|
=
\begin{cases}
2^r,
&
\text{if $o(g)=2$,}
\\
2^{r-1},
&
\text{if $o(g)=4$.}
\end{cases}
\end{equation}
Observe that this equality holds for every group $G$ in $\FF_1$, not only for non-Dedekind groups.
Now, if $G$ is not a Dedekind group, then according to the previous paragraph, every
$g\in G\smallsetminus A$ generates a non-normal subgroup of $G$.
Consequently, $\mci^*(G)$ is as given in (\ref{mci* for F1}).

Let us now obtain $\mni^*(G)$, which by Proposition \ref{mni=mni*} coincides with
$\mni(G)$.
Let  again $g$ be an arbitrary element of $G\smallsetminus A$, and put
$N=\langle g^2 \rangle$, which is a normal subgroup of $G$.
Then
\[
|N_G(\langle g \rangle):\langle g \rangle|
=
|N_{G/N}(\langle gN \rangle):\langle gN \rangle|
=
|C_{G/N}(gN):\langle gN \rangle|,
\]
since $gN$ has order $2$ in $G/N$.
By applying (\ref{|C_G(g):<g>| for F1}), which is valid for every group in $\FF_1$, to the group $G/N$, we get
\[
|N_G(\langle g \rangle):\langle g \rangle|
=
2^{d(A/N)}
=
\begin{cases}
2^r,
&
\text{if $g^2\in A^2$,}
\\
2^{r-1},
&
\text{if $g^2\not\in A^2$.}
\end{cases}
\]
Now, by (i) of Lemma \ref{g^2}, we have $g^2\in A^2$ or $g^2\not\in A^2$ simultaneously for every $g\in G\smallsetminus A$, according as $b^2\in A^2$ or not.
This proves (\ref{mni* for F1}).
\end{proof}

\begin{rmk}
If $G\in\FF_1$ and $b^2\not\in A^2$, then $\langle b^2 \rangle$ is a direct factor of $A$,
(for this, we need to use that $b^2$ is of order $2$).
Then $G$ can be given as a semidirect product $G=\langle b \rangle \ltimes \widetilde A$, where $d(\widetilde A)=r-1$ and $\widetilde a^b=\widetilde a^s$ for every
$\widetilde a\in \widetilde A$.
Thus the groups of this kind are a generalisation of the groups given in
(\ref{k=1 1st infinite family}) and (\ref{k=1 2nd infinite family}), which were infinite families of $2$-groups satisfying $\mni(G)=2$.
\end{rmk}

\begin{thm}
\label{values for F2}
Let $G$ be a $2$-group in the family $\FF_2$.
If the order of $A^*$ is $2^m$, then
\begin{equation}
\label{mci* for F2}
\mci^*(G)
=
\begin{cases}
2^{m+1},
&
\parbox[t]{.55\textwidth}
{if $A^*$ is elementary abelian, and $G\smallsetminus A$ contains an element of order $2$,}
\\[15pt]
2^m,
&
\text{otherwise,}
\end{cases}
\end{equation}
and, if $n\ge 3$,
\begin{equation}
\label{mni* for F2}
\mni(G)
=
\mni^*(G)
=
\begin{cases}
2^{m+1},
&
\parbox[t]{.45\textwidth}
{if $A^*$ is elementary abelian, and $b^2\in A^2$ or $b^2z\in A^2$,}
\\[15pt]
2^m,
&
\text{otherwise.}
\end{cases}
\end{equation}
\end{thm}

\begin{proof}
Let us first obtain the value of $\mci^*(G)$.
Let $g$ be an arbitrary element of $G\smallsetminus A$.
Since $C_A(b)=\Omega_1(A)$, we can argue as in the proof of
Theorem \ref{values for F1} to get
\begin{equation}
\label{|C_G(g):<g>| for F2}
|C_G(g):\langle g \rangle|
=
\begin{cases}
2^r,
&
\text{if $o(g)=2$,}
\\
2^{r-1},
&
\text{if $o(g)=4$,}
\end{cases}
\end{equation}
where $r$ is the rank of the abelian group $A$.
On the other hand, if $g\in A$ and $\langle g \rangle \not\trianglelefteq G$ (there is at least one such element, namely $a_1$) then necessarily $o(g)=2^n$.
Since $C_G(g)=A$, it follows that
\[
|C_G(g):\langle g \rangle| = 2^m.
\]
Now, since $r\le m+1$, with equality if and only if $A^*$ is elementary abelian, it readily follows that $\mci^*(G)$ is as in (\ref{mci* for F2}): simply observe that, if there exists
$g\in G\smallsetminus A$ of order $2$, then $\langle g \rangle$ is not normal in $G$.

Let us now calculate $\mni^*(G)$, under the assumption that $n\ge 3$.
In this case, every $g\in G\smallsetminus A$ generates a non-normal subgroup of $G$, since $o(g)\le 4$ and $[g,G]$ contains $a_1^{s-1}z$, which is of order $2^{n-1}$.
Put $N=\langle g^2 \rangle$, which is a normal subgroup of $G$.
Then, as in the proof of Theorem \ref{values for F1}, we have
\[
|N_G(\langle g \rangle):\langle g \rangle|
=
|C_{G/N}(gN):\langle gN \rangle|.
\]
Now observe that $G/N$ is a group either in $\FF_1$ or in $\FF_2$ (depending on where $N$ is located inside $A$).
Thus by applying (\ref{|C_G(g):<g>| for F1}) or (\ref{|C_G(g):<g>| for F2}), it follows that
\[
|N_G(\langle g \rangle):\langle g \rangle|
=
2^{d(A/N)}
=
\begin{cases}
2^r,
&
\text{if $g^2\in A^2$,}
\\
2^{r-1},
&
\text{if $g^2\not\in A^2$.}
\end{cases}
\]
On the other hand, if $g\in A$ and $\langle g \rangle \not\trianglelefteq G$ then
$o(g)=2^n$, and
\begin{equation}
\label{normalizer for g in A}
|N_G(\langle g \rangle):\langle g \rangle|=2^m.
\end{equation}
Consequently,
\[
\mni^*(G)
=
\begin{cases}
2^{m+1},
&
\parbox[t]{.6\textwidth}
{if $A^*$ is elementary abelian, and there exists $g\in G\smallsetminus A$ such that
$g^2\in A^2$,}
\\[15pt]
2^m,
& \text{otherwise.}
\end{cases}
\]
Then (\ref{mni* for F2}) follows from here, since by Lemma \ref{g^2}, $g^2$ is congruent to $b^2$ or $b^2z$ modulo $A^2$ for every $g\in G\smallsetminus A$, and both cases occur.
\end{proof}

\begin{rmk}
Formula (\ref{mni* for F2}) is not valid for $n=2$.
Let us consider the group
\[
G = \langle a_1,a_2,b
\mid a_1^4=a_2^2=1,\ b^2=a_1^2,\ a_1^b=a_1^{-1}a_2,
\ [a_2,b]=[a_1,a_2]=1 \rangle,
\]
of order $16$.
Then $G$ belongs to $\FF_2$, with $n=2$, $m=1$, $s=-1$, and $z=a_2$.
Since every element $g\in G\smallsetminus A$ is of order $4$ by Lemma \ref{g^2}, it follows that $|N_G(\langle g \rangle):\langle g \rangle|=2$ if
furthermore $\langle g \rangle \not\trianglelefteq G$.
This, together with (\ref{normalizer for g in A}), proves that $\mni^*(G)=2$, which does not match the value given by (\ref{mni* for F2}), which is $4$ in this case.
\end{rmk}

\begin{rmk}
Observe that the value of $\mci^*(G)$ for both families $\FF_1$ and $\FF_2$ depends on the existence of an element of order $2$ in the difference $G\smallsetminus A$.
One can easily describe when such an element exists in terms of the defining parameters of the groups in question.
More precisely, if $G$ belongs to $\FF_1$ then $G\smallsetminus A$ contains an element of order $2$ if and only if either $s=-1$ and $b^2=1$, or $s=-1+2^{n-1}$ and $b^2\in A^{2^{n-1}}$.
On the other hand, if $G$ belongs to $\FF_2$ then there is an element of order $2$ in
$G\smallsetminus A$ if and only if either $s=-1$ and $b^2=1$ or $z$, or $s=-1+2^{n-1}$ and $b^2\in (A^*)^{2^{n-1}}$ or $b^2\in a_1^{2^{n-1}} z(A^*)^{2^{n-1}}$.
\end{rmk}

We conclude by proving Theorem B.
The following lemma will be needed.

\begin{lem}
\label{mci* in quotients}
Let $G$ be a finite group, and let $N$ be a normal subgroup of $G$ for which $G/N$ is not a Dedekind group.
Then $\mci^*(G/N)\le |N|\mci^*(G)$.
\end{lem}

\begin{proof}
If $g\in G$ then $|C_{G/N}(gN)|\le |C_G(g)|$ (this can be seen by looking at the conjugacy classes), and $|\langle gN \rangle|\ge |\langle g \rangle|/|N|$.
The result follows.
\end{proof}

Observe that the inequality $\mci^*(G/N)\le \mci^*(G)$ need not hold in general.
For example, if $G\cong Q_{2^n}$ with $n\ge 4$, then $\mci^*(G)=1$ but
$\mci^*(G/Z(G))=2$.

\begin{thm}
Let $G$ be a non-Dedekind finite $2$-group, and suppose that either
$\mni(G)=2^k$, $\mni^*(G)=2^k$, or $\mci^*(G)=2^k$.
Then there exists a polynomial function $f(k)$ of degree four such that,
if $|G|>2^{f(k)}$, then $G$ belongs to one of the families $\FF_1$ or $\FF_2$.
\end{thm}

\begin{proof}
It suffices to prove the result under the condition that $\mci^*(G)=2^k$.
By Lemma \ref{mci^*(G)=1}, we may assume that $k\ge 1$, since generalised quarternion groups belong to $\FF_1$.
Let $A$ be a maximal abelian normal subgroup of $G$.
We split the proof of the theorem into two cases, according as all
subgroups of $A$ are normal in $G$, or not.

(i)
Assume first that every subgroup of $A$ is normal in $G$.
In particular, we have $\Omega_1(A)\le Z(G)$.
Let us write $A=\langle a_1 \rangle \times \cdots \times \langle a_r \rangle$,
where $o(a_1)\ge \cdots \ge o(a_r)$, and all factors are non-trivial.
We put $o(a_1)=2^n$, so that $\exp A=2^n$.
For simplicity, we write $U$ for the group of units $\UU(\Z/2^n\Z)$.

If $\langle x \rangle$ is a non-normal cyclic subgroup of $G$ then, by using that
$\mci^*(G)=2^k$, we get
\[
2^k \ge |C_G(x):\langle x \rangle|
\ge |\Omega_1(A)\langle x \rangle:\langle x \rangle|
= |\Omega_1(A):\Omega_1(A)\cap \langle x \rangle| \ge 2^{r-1},
\]
and consequently $r\le k+1$.
If we also have $n\le k+1$ then $|A|\le 2^{k^2+2k+1}$, and by Corollary 2 to Theorem 1.17 of \cite{suz} we get $|G|\le 2^{f_1(k)}$ for a polynomial $f_1(k)$ of degree $4$.
Thus we may assume that $n\ge k+2$ in the sequel.

Let $g$ be an arbitrary element of $G$.
If $s$ is an integer such that $a_1^g=a_1^s$, then we have $a^g=a^s$
for every $a\in A$, by Proposition \ref{when all normal}.
Note that this property implies that $C_G(a_1)=C_G(A)=A$.
Let $q$ be the order of $\overline s$ in $U$.
Since
\[
a_1^{g^j} = a_1^{s^j}
\quad
\text{for every $j\ge 1$,}
\]
it follows that $g^q$ is the first power of $g$ lying in $A$.
Consequently $|\langle g \rangle:\langle g \rangle \cap A|=q$.

Assume now that $g\in G\smallsetminus A$, so that $q>1$.
Since $q$ is a power of $2$, we can consider the element $g_1=g^{q/2}$.
Then $a_1^{g_1}=a_1^{s_1}$, where $s_1=s^{q/2}$ is such that
$o(\overline s_1)=2$ in $U$.
Since $n\ge k+2\ge 3$, there are three possibilities for $\overline{s_1}$:
it can be $\overline{1+2^{n-1}}$, $\overline{-1+2^{n-1}}$ or $\overline{-1}$.

We claim that $\overline{s_1}\ne \overline{1+2^{n-1}}$.
Assume otherwise, and put $J=\langle a_1,g_1 \rangle$.
Observe that $J$ is not a Dedekind group, since $\exp J\ge 8$.
Also, $J$ is not a group of type (R1) or (R2) in Blackburn's classification of $2$-groups with $R(G)\ne 1$, since a group of any of those types needs at least $3$ generators.
Finally, $J$ cannot be either of type (R3), i.e.\ a $Q$-group, since the centre of a
$Q$-group is elementary abelian (as happens with all non-abelian groups in the family
$\FF_1$), and $a_1^2\in Z(J)$ is of order at least $4$.
Consequently, $R(J)=1$ and, by Lemma \ref{cyclic avoiding R(G)}, there exists a cyclic non-normal subgroup $H$ of $J$ such that $\langle a_1 \rangle \cap H=1$.
Since $a_1^2\in Z(J)$, we have
\[
2^k \ge |C_G(H):H| \ge |\langle a_1^2 \rangle|.
\]
Thus $o(a_1)\le 2^{k+1}$ and $n\le k+1$, which is a contradiction.
Hence either $\overline{s_1}=\overline{-1}$ or $\overline{-1+2^{n-1}}$.
Now these two values are not squares in $U$, while $s_1=s^{q/2}$ and $q/2$
is a power of $2$.
This implies that $q=2$.
Consequently $g^2\in A$, $g=g_1$ and $s=s_1$.
Thus either $a^g=a^{-1}$ for every $a\in A$, or $a^g=a^{-1+2^{n-1}}$
for every $a\in A$.

According to this last property, the image of the embedding
$\varphi:G/A\rightarrow U$ of Proposition \ref{when all normal}
lies in the subgroup $V=\langle \overline{-1},\overline{-1+2^{n-1}} \rangle$.
If $\imm \varphi=V$ then there is an element $g\in G\smallsetminus A$
such that $a_1^g=a_1^{1+2^{n-1}}$, which is impossible as shown in the
last paragraph.
Hence $\imm \varphi$ is either $\langle \overline{-1} \rangle$ or
$\langle \overline{-1+2^{n-1}} \rangle$, and consequently $|G:A|=2$.
If we choose an element $b\in G\smallsetminus A$, then $b^2\in Z(G)=\Omega_1(A)$.
We conclude that $G$ lies in the family $\FF_1$.

(ii)
Assume now that there are subgroups of $A$ which are not normal in $G$.
According to Proposition \ref{when all normal}, there is a direct factor
of $A$ which is not normal in $G$.
So we can write
$A=\langle a_1 \rangle \times \cdots \times \langle a_r \rangle$,
where $\langle a_1 \rangle$ is not normal in $G$, and all factors are non-trivial.
(Unlike in case (i), now there is no relation between the orders of the
elements $a_i$.)
Since $A$ is normal in $G$, we necessarily have $r\ge 2$.

Let $A^*=\langle a_2,\ldots,a_r \rangle$.
Since $|C_G(a_1):\langle a_1 \rangle|\ge |A^*|$, it follows
that $|A^*|\le 2^k$.
In particular, we have $r\le k+1$, as in case (i).
Also, if $o(a_1)=2^n$ then we may assume that $n\ge 2k+2$, since otherwise
$|A|\le 2^{3k+1}$ and we get, as in case (i), that $|G|\le 2^{f_2(k)}$ for some polynomial function $f_2(k)$, of degree $2$ in this case.
In particular, $n$ is at least $4$.

If $H$ is a cyclic subgroup of $A$ of order at most $2^{n-k}$ then
\[
|C_G(H):H| \ge |A:H| \ge 2^{n+1}/2^{n-k} =2^{k+1},
\]
and consequently $H$ is normal in $G$.
It follows that the subgroups
\[
\langle a_1^{2^k} \rangle, \langle a_2 \rangle, \ldots, \langle a_r \rangle,
\langle a_1^{2^k}a_2 \rangle, \ldots, \langle a_1^{2^k}a_r \rangle
\]
are all normal in $G$.
By Proposition \ref{when all normal}, every subgroup of
$\langle a_1^{2^k},a_2,\ldots,a_r \rangle$ is normal in $G$.
In particular, $A^*$ is normal in $G$.

Let $C$ be defined by the condition $C/A^*=C_{G/A^*}(A/A^*)$.
If $c\in C$ then $[a_1,c]\in A^*$ has order at most $2^k$, and consequently
$[a_1^{2^k},c]=1$.
In other words, $\langle a_1^{2^k} \rangle$ is contained in $Z(C)$.
If $C$ is not a Dedekind group, then $|R(C)|\le 2$, and by
Lemma \ref{cyclic avoiding R(G)}, there exists a non-normal cyclic subgroup $H$ of $C$ such that $|H\cap \langle a_1 \rangle|\le 2$.
Since
\[
|C_G(H):H| \ge |H\langle a_1^{2^k} \rangle:H|
= |\langle a_1^{2^k} \rangle:H\cap \langle a_1^{2^k} \rangle|
\ge |\langle a_1^{2^k} \rangle|/2 = 2^{n-k-1},
\]
it follows that $n\le 2k+1$, contrary to our assumption above.
Hence $C$ is a Dedekind group.
We cannot have $C\cong Q_8\times E$, with $E$ elementary abelian, since
$\exp C\ge \exp A=2^n\ge 2^4$.
Thus we are only left with the case that $C$ is abelian.
Then $C=A$, since $A$ is a maximal abelian normal subgroup of $G$.
Consequently, we have $C_{G/A^*}(A/A^*)=A/A^*$, which means that $A/A^*$
is a maximal abelian normal subgroup of $G/A^*$.
Also, since $A/A^*=\langle a_1A^* \rangle$ is cyclic, every subgroup of
$A/A^*$ is normal in $G/A^*$.
If $G/A^*$ is abelian then $G/A^*$ belongs to the family $\FF_1$, and otherwise we are
in the situation of case (i).
It follows that either $|G/A^*|\le 2^{f_1(2k)}$ or $G/A^*$ lies in $\FF_1$.
(Take into account that $\mci^*(G/A^*)\le |A^*|\mci^*(G)\le 2^{2k}$ by
Lemma \ref{mci* in quotients}.)

In the former case, we have $|G|\le 2^{f_1(2k)+k}$.
In the latter, we get $|G:A|=2$, and if we choose $b\in G\smallsetminus A$ then
$a_1^b=a_1^sz$, with either $s=-1$ or $-1+2^{n-1}$, and $z\in A^*$ different from $1$.
It follows that
\[
(a_1^{2^k})^b = (a_1^{2^k})^s.
\]
Since all subgroups of $\langle a_1^{2^k},a_2,\ldots,a_r \rangle$ are
normal in $G$, and since $o(a_1^{2^k})\ge o(a_2),\ldots,o(a_r)$, it follows from
Proposition \ref{when all normal} that $a_i^b=a_i^s$ for every $i=2,\ldots,r$.
Thus $(a^*)^b=(a^*)^s$ for every $a^*\in A^*$.
Hence $Z(G)=\Omega_1(A)$ and, in particular, $b^2\in\Omega_1(A)$.
Now observe that
\[
a_1 = a_1^{b^2} = (a_1^s z)^b = (a_1^s z)^s z^b
= a_1^{s^2} z^sz^b = a_1z^{-1} z^b,
\]
since $o(z)\le 2^k\le 2^{n-1}$.
It follows that $z^b=z$, and so $z\in Z(G)$.
Thus $z\in\Omega_1(A^*)$.
We conclude that $G$ lies in the family $\FF_2$.

\vspace{8pt}

Now, by bringing together the results obtained in (i) and (ii), it follows that there is a polynomial $f(k)$ of degree $4$ such that either $|G|\le 2^{f(k)}$ or $G$ belongs to one of the families $\FF_1$ or $\FF_2$, as desired.
\end{proof}

\vspace{20pt}

\noindent
\textit{Acknowledgment\/}.
We thank R. Esteban-Romero for drawing our attention to the reference \cite{zha-guo}.

\end{document}